\documentclass[review,onefignum,onetabnum]{siamonline190516}

\usepackage{lipsum}
\usepackage{mathrsfs}
\usepackage{amsfonts}
\usepackage{graphicx}
\usepackage{epstopdf}
\usepackage{algorithmic}
\ifpdf
  \DeclareGraphicsExtensions{.eps,.pdf,.png,.jpg}
\else
  \DeclareGraphicsExtensions{.eps}
\fi

\usepackage{enumitem}
\setlist[enumerate]{leftmargin=.5in}
\setlist[itemize]{leftmargin=.5in}

\newsiamremark{remark}{Remark}
\newsiamremark{hypothesis}{Hypothesis}
\crefname{hypothesis}{Hypothesis}{Hypotheses}
\newsiamthm{claim}{Claim}

\title{Absence of the analytic continuation of elastic transmission eigenfunctions at rectangular corners \thanks{Submitted to the editors DATE.
}
}

\author{Jianli Xiang\thanks{Three Gorges Mathematical Research Center, College of Mathematics and Physics, China Three Gorges University, Yichang 443002, People's Republic of China(\email{xiangjianli@ctgu.edu.cn}).}\and
Guanghui Hu \thanks{Corresponding author: School of Mathematical Sciences and LPMC, Nankai University, Tianjin 300071, People's Republic of China (\email{ghhu@nankai.edu.cn}).}
}

\usepackage{amsopn}

\usepackage{color}
\usepackage{amssymb}

\newcommand{\be}{\begin{eqnarray}}
\newcommand{\ben}{\begin{eqnarray*}}
\newcommand{\en}{\end{eqnarray}}
\newcommand{\enn}{\end{eqnarray*}}

\ifpdf
\hypersetup{
  pdftitle={Inverse medium scattering problem},
  pdfauthor={Jianli Xiang}
}
\fi
\definecolor{rot}{rgb}{1,0,0}

\nolinenumbers
\begin{document}

\maketitle

\begin{abstract}
  We study time harmonic scattering problems in linear elasticity in $\mathbb{R}^{2}$. We show that certain penetrable scatterers with rectangular corners scatter every incident wave nontrivially. Even though these scatterers have interior transmission eigenvalues, the far field operator has a trivial kernel at every real frequency. Our approach relies on a special decomposition of the elastic Lam\'e operator and also provides an alternative idea for treating inverse elastic medium problems with a general polygonal support.
\end{abstract}



\section{Introduction}

Recently, the scattering problems for elastic waves have received increasing attention because of the significant applications in geophysics and seismology (\cite{Ammari2015, Bonnet2005,Landau1986}), and the inverse elastic scattering problems are of great interest because they describe many geometrical and physical identification applications concerning deformable bodies. The propagation of elastic waves is governed by the Navier equation, which is complex due to the coupling of the compressional and shear waves with different wavenumbers. The inverse problem is to find the shape and location of an obstacle by measuring the scattered waves, close or far from the scatterer. In the theory of elastic waves, there is a great wealth of results on the inverse problem by the linear sampling method \cite{Arens2001, CGK2002,GM2007,PS2007}, the factorization method \cite{AK2002,CKA2007,Ser2005,XY2022}, the enclosed method \cite{MI2007}, a gradient descent method \cite{Louer2015,PK2000}, the probing method \cite{KS2015} or the direct sampling method \cite{LLS2023}.

For inverse medium scattering, there appears to an important interior transmission problem (ITP), which is formulated as a boundary value problem for a pair of governing field equations over a bounded domain $D\subseteq\mathbb{R}^{2}$ representing the support of a scatterer that are coupled through the Cauchy data on $\partial D$.
Assuming time-harmonic plane wave propagation and scattering, the existence of a non-trivial solution to the ITP with homogeneous boundary conditions amounts to that of an incident wave field, illuminating inhomogeneity $D$, that generates no scattered field. Values of frequencies $\omega$ at which this phenomenon occurs are the so-called interior transmission eigenvalues (ITE), and their role is important since the solvability of the far-field equation is problematic and also many reconstruction algorithms in inverse scattering will not work correctly at these frequencies \cite{CGK2002, CKA2007,GM2007,PS2007, Ser2005,XY2022}.

The ITP was first discussed by Colton, Kirsch and Monk in the mid 1980s in connection with the inverse scattering problem for acoustic waves in an inhomogeneous medium \cite{CM1988,Kir1986}. Since that time the ITP has come to play a basic role in inverse scattering for both acoustic, electromagnetic and elastic waves. The ITP is a type of nonelliptic and non-self-adjoint problem, so its study is mathematically interesting and challenging. In 2002, the discreteness of the ITE for the non absorbing, anisotropic and inhomogeneous elasticity was investigated by Charalambopoulos \cite{Char2002}, in which the ITP was represented as a superposition of an injective operator and a compact perturbation by constructing a corresponding modified ITP. The conditions in \cite{Char2002} imposed on the entries of the tensors are rather complicated and their reduction to the isotropic case is non-obvious. In order to clarify this point, the paper \cite{Char2008} made an appropriate decomposition of the corresponding far-field operator, established an isomorphism property for the boundary integral operator and finally stated a sufficient condition on the elastic parameters of the inclusion and the host environment for the discreteness of the set formed by the ITE. In 2010, Bellis and Guzina \cite{BG2010} considered a piecewise-homogeneous, elastic or viscoelastic obstacle in a likewise heterogeneous background solid and they had proved the discreteness as well as the nonexistence case of ITE. In 2013, the existence of the ITE was presented in \cite{BCG2013} for the first time. So far, many results have focused mostly on the spectral properties (existence, discreteness, infiniteness and Weyl's laws) of the ITE.

In fact, the transmission eigenfunctions reveal certain distinct and intriguing features about the scatterers.
Then some results about the eigenfunctions themselves started surfacing, in the form of completeness of the eigenfunctions in some sense.
Under certain regularity assumptions, geometric structures of transmission eigenfunctions are discovered.
In \cite{B2018,BL2017,BL2021,BLL2017}, it had shown that eigenfunctions with some regularity vanish at corners, edges or high-curvature places of the support of an inhomogeneous index of refraction.
For radial geometry, the boundary-localizing properties were exhibited in the acoustic transmission eigenfunctions \cite{CDH2023,CDL2023,JL2023}, and further be extended to the elastic transmission eigenfunctions in \cite{JLZ2023}.
For an $n$-dimensional rectangle, Blasten \cite{BPS2014} proved that the transmission eigenfunctions for the Helmholtz equation cannot be analytically extended across the boundary based on the use of complex geometric optics (CGO) solutions.
Then the CGO approach was extended to the cases of a convex corner in $\mathbb{R}^{2}$ and a circular conic corner in $\mathbb{R}^{3}$ whose opening angle is outside of a countable subset of $(0,\pi)$ \cite{PSV2017}.
For more general corners in $\mathbb{R}^{2}$ and edges in $\mathbb{R}^{3}$ with an arbitrary angle in the case of piecewise analytic potentials \cite{EH2015}, the failure of the analytic continuation of transmission eigenfunctions was proved by the expansion of solutions to the Helmholtz equation and also lead to uniqueness in shape identification of a convex penetrable obstacle of polygonal or polyhedral type by a single far-field pattern.
Furthermore, the arguments were achieved for curvilinear polygons in $\mathbb{R}^{2}$, curvilinear polyhedra and circular cones in $\mathbb{R}^{3}$ with an arbitrary piecewise H\"older continuous potential \cite{EH2018}, that the approach relied on the corner singularity analysis of solutions to the inhomogeneous Laplace equation in weighted H\"older spaces.
However, the existing literature on the elastic transmission eigenfunctions is relatively scarce.
In \cite{BL2018}, the vanishing property at corners and edges for smooth transmission eigenfunctions in elasticity was shown by constructing a new type of exponential solution for the Navier equation.

The aim of this paper is to verify the failure of the analytic continuation of elastic transmission eigenfunctions.
Different from the CGO-solutions method, our analysis relies on the diagonalized decomposition of the Lam\'e operator that are of independent interest.
Moreover, we obtain a uniqueness result for the inverse medium scattering problem with a single
incoming wave within the class of convex polygons whose internal or external corners are standard right angles.

The rest of the paper is organized as follows. In section \ref{sec2}, mathematical formulations and main results are presented for the elastic interior transmission eigenfunctions.
In section \ref{sec3}, we derive our crucial auxiliary result on the transmission problem for the Navier equation. At the end of section \ref{sec4}, we verify Theorem \ref{Main} and Corollaries \ref{cor1}, \ref{cor2} relying on Lemma \ref{mainlem}.

\section{Mathematical formulation and main results} \label{sec2}
In this paper, we consider the inverse medium scattering problem of time-harmonic elastic waves. Assume that the Lam\'e constants $\lambda$ and $\mu$ satisfying $\mu>0$, $\mu+\lambda>0$. The scattering problem we are dealing with is now modeled by the following Navier equation
\begin{equation} \label{a}
\mu\Delta u+(\lambda+\mu)\nabla(\nabla\cdot u)+\omega^{2}\rho(x)u=0, \quad {\rm in} \quad \mathbb{R}^{2}.
\end{equation}
Here, $u=u^{{\rm in}}+u^{{\rm sc}}$ is the total displacement field which is the superposition of the given incident plane wave $u^{{\rm in}}$ and the scattered wave $u^{{\rm sc}}$. The circular frequency $\omega>0$ and the density function $\rho(x)\in L^{\infty}(\mathbb{R}^{2})$ satisfies $\rho(x)=1$ in $\mathbb{R}^{2}\backslash\overline{D}$.

By the Helmholtz decomposition theorem, the scattered field $u^{{\rm sc}}$ can be decomposed as $u^{{\rm sc}}=u^{{\rm sc}}_{p}+u^{{\rm sc}}_{s}$, where $u^{{\rm sc}}_{p}$ denotes the compressional wave and $u^{{\rm sc}}_{s}$ denotes the shear wave, $k_{p}$ is the compressional wave number and $k_{s}$ is the shear wave number. They are given by the following forms respectively:
\begin{equation*}
u^{{\rm sc}}_{p}=-\frac{1}{k_{p}^{2}}{\rm grad}\,{\rm div}\, u^{{\rm sc}},\quad k_{p}=\frac{\omega}{\sqrt{2\mu+\lambda}}, \quad
u^{{\rm sc}}_{s}=\frac{1}{k_{s}^{2}}\overrightarrow{{\rm curl}}\, {\rm curl}\, u^{{\rm sc}},\quad k_{s}=\frac{\omega}{\sqrt{\mu}},
\end{equation*}
and they satisfy $\Delta u^{{\rm sc}}_{p}+k_{p}^{2}u^{{\rm sc}}_{p}=0$ and $\Delta u^{{\rm sc}}_{s}+k_{s}^{2}u^{{\rm sc}}_{s}=0$ in $\mathbb{R}^{2}\backslash\overline{D}$. In addition, the Kupradze radiation condition is required to the scattered field $u^{{\rm sc}}$, i.e.
\begin{equation}  \label{rad}
\mathop {\lim }\limits_{r \to \infty }\sqrt[]{r} \Big(\frac{\partial u^{{\rm sc}}_{p}}{\partial r}-ik_{p}u^{{\rm sc}}_{p}\Big)=0,~\quad~
\mathop {\lim }\limits_{r \to \infty }\sqrt[]{r} \Big(\frac{\partial u^{{\rm sc}}_{s}}{\partial r}-ik_{s}u^{{\rm sc}}_{s}\Big)=0, \quad r=|x|.
\end{equation}
And the radiation condition (\ref{rad}) is assumed to hold in all directions $\hat{x}=x/|x|\in\mathbb{S}$, here $\mathbb{S}:=\{x\in \mathbb{R}^2:|x|=1\}$ denotes the unit circle in $\mathbb{R}^{2}$.

It is well known that the radiating solution to the Navier equation has the following asymptotic behavior:
\begin{equation*}
u^{{\rm sc}}(x)=\frac{e^{ik_{p}r}}{\sqrt{r}}u_{p}^{\infty}(\hat{x}) +\frac{e^{ik_{s}r}}{\sqrt{r}}u_{s}^{\infty}(\hat{x})
+\mathcal{O}\left(\frac{1}{r^{3/2}}\right),\quad {\rm as} \quad r\rightarrow+\infty,
\end{equation*}
where $u_{p}^{\infty}(\hat{x})\parallel\hat{x}$ and $u_{s}^{\infty}(\hat{x})\perp\hat{x}$. The functions $u_{p}^{\infty}$, $u_{s}^{\infty}$ are known as compressional and shear far field pattern of $u^{{\rm sc}}$, respectively,
\begin{equation*}
u^{\infty}(\hat{x})=u_{p}^{\infty}(\hat{x})+u_{s}^{\infty}(\hat{x}).
\end{equation*}

For notational convenience in the following text, we set $\mathcal{L}:=\mu\Delta+(\lambda+\mu) \nabla \nabla\cdot$, and the incident wave $u^{{\rm in}}$ satisfies the equation:
\begin{equation} \label{b}
\mathcal{L}u^{{\rm in}}+\omega^{2}u^{{\rm in}}=0, \quad {\rm in} \quad \mathbb{R}^{2},
\end{equation}
and therefore that
\begin{equation*}
\mathcal{L}u^{{\rm sc}}+\omega^{2}\rho u^{{\rm sc}}=\omega^{2}(1-\rho) u^{{\rm in}}, \quad {\rm in} \quad \mathbb{R}^{2}.
\end{equation*}
For penetrable elastic scattering problem we need to define the stress vector (or traction): on any curve in $\mathbb{R}^{2}$ with a normal $\nu$, we can define the stress vector $Tu$ by
\begin{equation*}
Tu:=2\, \mu\, \partial_\nu\, u+\lambda\, \nu\, {\rm div}\, u+\mu\, \nu \times {\rm curl}\, u.
\end{equation*}
\begin{definition}
$\omega^{2}$ is called an interior transmission eigenvalue, if there are nontrivial $w$ and $v$ satisfying the interior transmission problem:
\begin{equation} \label{ITP}
\left\{\begin{array}{ll}
\mathcal{L}v+\omega^{2}v=0, & {\rm in} \quad D,\\
\mathcal{L}w+\omega^{2}\rho w=\omega^{2}(1-\rho) v,  & {\rm in} \quad D,  \\
w=0, \quad  Tw=0,  & {\rm on} \quad \partial D.
\end{array}\right.
\end{equation}
\end{definition}
Note that in our applications we have $w=u^{{\rm sc}}$ and $v=u^{{\rm in}}$.
The results of \cite{BCG2013} guarantee that, the pair $(\rho,D)$ have infinitely many interior transmission eigenvalues; i.e., there are infinitely many positive values of $\omega^{2}$ for which there exist nontrivial solutions to \eqref{ITP}. These results do not guarantee that any eigenfunction pair $(u^{{\rm sc}},u^{{\rm in}})$ extends to $\mathbb{R}^{2}$. The scattered wave $u^{{\rm sc}}$ belongs to $[H_{0}^{2}(D)]^{2}$, and therefore always extends to all of $\mathbb{R}^{2}$ as a function which is zero outside $D$. The waves $u^{{\rm in}}$, however, are only guaranteed to satisfy $u^{{\rm in}}\in[L^{2}(D)]^{2}$ and $\mathcal{L}u^{{\rm in}} \in [L^{2}(D)]^{2}$.
We will refer to $u^{{\rm in}}$ as an interior incident wave to emphasize that it is only defined inside $D$. A consequence of Theorem \ref{Main} below, is that, if $\rho$ satisfies the hypothesis of the theorem, and $D$ contains an open neighborhood of the corner point, then there are no nontrivial solutions to \eqref{ITP}, i.e., no non-scattering frequency for the pair $(\rho,D)$.

\begin{theorem} \label{Main}
Assume that $D\subset\mathbb{R}^{2}$ is a polygon whose internal or external angles are standard right angles and $\rho(x)|_{D}=\rho_{0}$ is a constant different from one (i.e. $\rho_{0}\neq1$). The pair $(u^{{\rm sc}},u^{{\rm in}})$ are interior transmission eigenfunctions of $\rho$ in $D$, i.e solutions to \eqref{ITP}, then $u^{{\rm in}}$ cannot be extended as an incident wave (i.e a solution to \eqref{b}) to any open neighborhood of the corner.
\end{theorem}

If there exist a Herglotz wave function $u_{g}$ defined by
\begin{equation*}
u_{g}(x)=e^{-i\pi/4}\int_{\mathbb{S}}\left\{\sqrt{\frac{k_{p}}{\omega}} e^{ik_{p}d\cdot x}dg_{p}(d)+\sqrt{\frac{k_{s}}{\omega}}e^{ik_{s}d\cdot x}d^{\bot}g_{s}(d)\right\}{\rm d}s(d), \quad  g=(g_{p},g_{s})\in[L^{2}(\mathbb{S})]^{2},
\end{equation*}
such that the function pair $(u^{{\rm sc}}, u_{g})$ is a solution to \eqref{ITP}, then $\omega$ is called non-scattering frequency. A well known result states that the nullspace of the far-field operator $F$ is trivial if and only $\omega$ is not a non-scattering frequency.
\begin{corollary} \label{cor1}
A scatterer $(\rho,D)$ which satisfies the hypothesis of Theorem \ref{Main} has no non-scattering frequencies, or equivalently, the far-field operator $F$ is injective.
\end{corollary}

\begin{corollary} \label{cor2}
Assume that the scatterer $(\rho,D)$ satisfies the hypothesis of Theorem \ref{Main} and $D$ is convex, then $\partial D$ can be uniquely determined by a single far field pattern $u^{\infty}(\hat{x})$ for all $\hat{x}\in\mathbb{S}$.
\end{corollary}

\section{Preliminary lemmas} \label{sec3}
We introduce several notations before stating the main results. For $j\in\mathbb{N}_{0}:=\{0\}\cup\mathbb{N}$, $\nabla_{x}^{j}$ stands for the set of all partial derivatives of order $j$ with respect to the two components of $x=(x_{1},x_{2})^{\top}\in\mathbb{R}^{2}$, i.e.,
\begin{equation*}
\nabla_{x}^{j}u=\{\partial_{x_{1}}^{j_{1}} \partial_{x_{2}}^{j_{2}}:j_{1},j_{2}\in \mathbb{N}_{0},j_{1}+j_{2}=j \}.
\end{equation*}
Let $(r,\theta)$ be the polar coordinates of $(x_{1}, x_{2})^{\top}\in\mathbb{R}^{2}$, let $B_{R}$ denote the disk centered at the origin $O$ with radius $R$, and define
\begin{align*}
&\Gamma_{R}^{+}:=\{(r,\pi/2):0\leq r\leq R\}=\{(x_{1},x_{2}):x_{1}=0,~0\leq x_{2}\leq R\},\\
&\Gamma_{R}^{-}:=\{(r,0):0\leq r\leq R\}=\{(x_{1},x_{2}):x_{2}=0,~0\leq x_{1}\leq R\}.
\end{align*}

\begin{lemma} \label{lem1}
Let $\Gamma\subset\mathbb{R}^2$ be a Lipschitz surface. Then $u_1=u_2$ and $Tu_1=Tu_2$ on $\Gamma$ implies that $\partial_\nu u_1=\partial_\nu u_2$ on $\Gamma$.
\end{lemma}
\begin{proof}
Set $w:=u_1-u_2$, then we have $w=Tw=0$ on $\Gamma$. Introduce the tangential G\"unter derivative (refer to \cite{BTV2020}):
\begin{equation*}
\mathcal{M}w=\partial_\nu w-\nu\, {\rm div}\,w+\nu \times {\rm curl}\, w.
\end{equation*}
The condition $w=0$ on $\Gamma$ implies that $\mathcal{M}w=0$ on $\Gamma$. Note that
\begin{equation*}
Tw=\mu\, \partial_\nu w+(\lambda+\mu)\, \nu\, {\rm div}\, w
=(\lambda+2\mu)\,\partial_\nu w+(\lambda+\mu)\, \nu \times {\rm curl}\,w .
\end{equation*}
Then we obtain that
\begin{equation*}
\nu \times Tw=\mu\, \nu \times \partial_\nu w, \quad \nu \cdot Tw=(\lambda+2\mu)\, \nu \cdot \partial_\nu w,
\end{equation*}
respectively. Therefore,
\begin{equation*}
\partial_\nu w=\frac{1}{\mu}(\nu \times Tw) \times \nu+\frac{1}{\lambda+2\mu}(\nu \cdot Tw)\nu=0 \quad {\rm on} \quad \Gamma,
\end{equation*}
i.e., $\partial_\nu u_1=\partial_\nu u_2$ on $\Gamma$.
\end{proof}

\begin{lemma} \label{lem2}
Suppose $A\in\mathbb{R}^{2\times 2}$ is an invertible square matrix and $U,V\in\mathbb{R}^2$ are column vectors, then
\begin{equation*}
{\rm det}(A+UV^{\top})=(1+V^{\top}A^{-1}U)\,{\rm det}(A).
\end{equation*}
\end{lemma}

\begin{lemma}   \label{lem3}
Let $L=\mu(\eta_{1}^{2}+\eta_{2}^{2})I+(\lambda+\mu) \Theta\Theta^{\top}$ with $\Theta=(\eta_{1},\eta_{2})^{\top}$ and $\eta_{1}^{2}+\eta_{2}^{2}\neq0$, $I$ is the identity matrix in $\mathbb{R}^{2\times 2}$. Then

(i) ${\rm det}(L)=(\lambda+2\mu)\,\mu\,(\eta_{1}^{2}+\eta_{2}^{2})^2$.

(ii) The eigenvalues of $L$ are given by
\begin{equation*}
(\lambda+2\mu)(\eta_{1}^{2}+\eta_{2}^{2}),\quad \mu(\eta_{1}^{2}+\eta_{2}^{2}).
\end{equation*}

(iii) There exists an invertible matrix $P=\left(\begin{array}{cc}
-\eta_{2} & \eta_{1} \\
\eta_{1} & \eta_{2}
\end{array}\right)$ such that
\begin{equation*}
L=P\left(\begin{array}{cc}
\mu & 0 \\
0 & \lambda+2\mu
\end{array}\right)P,\quad\quad  P^{2}=(\eta_{1}^{2}+\eta_{2}^{2})\,I. 
\end{equation*}
\end{lemma}
\begin{proof}
(i) Set $A=\mu\, (\eta_{1}^{2}+\eta_{2}^{2})\,I$, $U=V=\sqrt{\lambda+\mu}\,\Theta$. Then
\begin{equation*}
{\rm det}(A)=\mu^{2}\, (\eta_{1}^{2}+\eta_{2}^{2})^{2}, \quad 1+V^{\top}A^{-1}U=1+\frac{\lambda+\mu}{\mu}.
\end{equation*}
Then the proof is completed by applying Lemma \ref{lem2}.

(ii) Set $\widetilde{L}=L-xI$, then
\begin{equation*}
{\rm det}(\widetilde{L})=\big[(\lambda+2\mu)(\eta_{1}^{2}+\eta_{2}^{2})-x\big] \big[\mu(\eta_{1}^{2}+\eta_{2}^{2})-x\big],
\end{equation*}
which completes the proof.

(iii) Since $\lambda+\mu>0$, we know that the two eigenvalues of $L$ are different which implies that the matrix $L$ can be diagonalized. For each eigenvalue, calculate the corresponding eigenfunctions to obtain an invertible matrix $P$.
\end{proof}

Recall that $\mathcal{L}:=\mu\Delta+(\lambda+\mu) \nabla \nabla\cdot$, that is,
\begin{equation*}
\mathcal{L}=\mu\left(\partial_{x_{1}}^{2}+\partial_{x_{2}}^{2}\right)I
+(\lambda+\mu)\left(\begin{array}{c}
\partial_{x_{1}} \\  \partial_{x_{2}}
\end{array}\right)\left(\begin{array}{cc}
\partial_{x_{1}} & \partial_{x_{2}}
\end{array}\right).
\end{equation*}
Applying Lemma \ref{lem3} yields,
\begin{equation*}
\mathcal{L}=\left(\begin{array}{cc}
-\partial_{x_{2}} & \partial_{x_{1}} \\
\partial_{x_{1}} & \partial_{x_{2}}
\end{array}\right)\left(\begin{array}{cc}
\mu & 0 \\
0 & \lambda+2\mu
\end{array}\right)\left(\begin{array}{cc}
-\partial_{x_{2}} & \partial_{x_{1}} \\
\partial_{x_{1}} & \partial_{x_{2}}
\end{array}\right),
\end{equation*}
and then
\begin{equation*}
\mathcal{L}^{2}=\left(\partial_{x_{1}}^{2}+\partial_{x_{2}}^{2}\right)\left(\begin{array}{cc}
-\partial_{x_{2}} & \partial_{x_{1}} \\
\partial_{x_{1}} & \partial_{x_{2}}
\end{array}\right)\left(\begin{array}{cc}
\mu^{2} & 0 \\
0 & (\lambda+2\mu)^{2}
\end{array}\right)\left(\begin{array}{cc}
-\partial_{x_{2}} & \partial_{x_{1}} \\
\partial_{x_{1}} & \partial_{x_{2}}
\end{array}\right)
=\Delta P_{1}(\partial)P_{2}(\partial),
\end{equation*}
where
\begin{equation*}
P_{1}(\partial):=\left(\begin{array}{cc}
-\partial_{x_{2}} & \partial_{x_{1}} \\
\partial_{x_{1}} & \partial_{x_{2}}
\end{array}\right)\left(\begin{array}{cc}
\mu & 0 \\  0 & \lambda+2\mu
\end{array}\right)  \quad
P_{2}(\partial):=\left(\begin{array}{cc}
\mu & 0 \\   0 & \lambda+2\mu
\end{array}\right)\left(\begin{array}{cc}
-\partial_{x_{2}} & \partial_{x_{1}} \\
\partial_{x_{1}} & \partial_{x_{2}}
\end{array}\right).
\end{equation*}

\begin{lemma} \label{mainlem}
Suppose the Lam\'e parameters $\mu$, $\lambda$ and $q_{1}$, $q_{2}$ are all constants in $B_{R}$ satisfying $\mu>0$, $\lambda+\mu>0$ and $q_{1}\neq q_{2}$. If the solution pair $(w,u_{0})$ solves the coupling problem
\begin{equation} \label{eq}
\left\{\begin{array}{ll}
\mathcal{L}w+q_{1}\omega^{2}w=\omega^{2}(q_{2}-q_{1})u_{0} & \quad {\rm in} \quad B_{R},  \\
\mathcal{L}u_{0}+q_{2}\omega^{2}u_{0}=0  &\quad {\rm in} \quad B_{R},  \\
w=Tw=0 & \quad {\rm on} \quad \Gamma_{R}^{\pm}.
\end{array}\right.
\end{equation}
then $w=u_{0}\equiv 0$ in $B_{R}$.
\end{lemma}
\begin{proof}
Let $\tau_{1}$, $\nu_{1}$ ($\tau_{2}$, $\nu_{2}$) denote the unit tangential and normal vectors on $\Gamma_{R}^{\pm}$ directed into $B_{R}\backslash\overline{\Sigma_{R}}$, then $\nu_{1}=(0,-1)^{\top}$, $\tau_{1}=(1,0)^{\top}$, $\nu_{2}=(-1,0)^{\top}$, $\tau_{2}=(0,1)^{\top}$. It is easy to obtain that $\partial_{x_{1}}=\partial_{\tau_{1}} =-\partial_{\nu_{2}}$, $\partial_{x_{2}}=\partial_{\tau_{2}}=-\partial_{\nu_{1}}$, then we conclude that for all $m\in\mathbb{N}_{0}$,
\begin{equation} \label{eq5}
\nabla_{x}^{m}\subset{\rm span}\{\partial_{\tau_{1}}^{m},\partial_{\tau_{1}}^{m-1}\partial_{\tau_{2}}, \partial_{\tau_{1}}^{m-2}\partial_{\tau_{2}}^{2}, \cdots, \partial_{\tau_{1}}^{2}\partial_{\tau_{2}}^{m-2}, \partial_{\tau_{1}}\partial_{\tau_{2}}^{m-1},\partial_{\tau_{2}}^{m}\}.
\end{equation}

It can also be claimed that for all $m\in\mathbb{N}_{0}$, $m\geq2$,
\begin{equation} \label{eq6}
\nabla_{x}^{m-2}\mathcal{L}w\subset{\rm span}\{ \nabla_{x}^{m}w \}.
\end{equation}
Next, we prove by induction that $\nabla_{x}^{m}u_{0}=0$ for all $m\in\mathbb{N}_{0}$.

{\bf Step 1.} We show that
\begin{equation} \label{eq7}
w=0, \quad \nabla_{x}w=0 \quad {\rm at}\quad O.
\end{equation}
The result $w(O)=0$ follows immediately from $w=0$ on $\Gamma_{R}^{\pm}$. Applying Lemma \ref{lem1} and the boundary condition (\ref{eq}) we
get
\begin{equation}  \label{eq8}
\partial_{\nu_{1}}w=\partial_{\nu_{2}}w=0 \quad {\rm on} \quad \Gamma_{R}^{\pm}.
\end{equation}
Therefore, $\nabla_{x}w(O)=0$.

{\bf Step 2.} We show that
\begin{equation*}
\nabla_{x}^{2}w=0, \quad u_{0}=0,\quad \mathcal{L}u_{0}=0,
\quad \mathcal{L}^{2}w=0 \quad {\rm at}\quad O.
\end{equation*}

From the boundary condition (\ref{eq8}) we see that
\begin{equation*}
\partial_{\tau_{1}}^{2}w=\partial_{\tau_{1}}\partial_{\nu_{1}}w=0, \quad \partial_{\tau_{2}}^{2}w=\partial_{\tau_{2}}\partial_{\nu_{2}}w=0, \quad
\partial_{\tau_{1}}\partial_{\tau_{2}}w=-\partial_{\tau_{1}}\partial_{\nu_{1}}w=0,
\end{equation*}
implying that $\nabla_{x}^{2}w=0$ at $O$ due to the relation (\ref{eq5}). Then $\nabla_{x}^{2}w(O)=0$. We obtain from (\ref{eq}),(\ref{eq6}) and (\ref{eq7}) that $u_{0}(O)=0$. Using (\ref{eq}) yields $\mathcal{L}u_{0}(O)=0$. Applying this condition into (\ref{eq}) we get $\mathcal{L}^{2}w(O)=0$.

{\bf Step 3.} We show that
\begin{equation*}
\nabla_{x}^{3}w=0, \quad \nabla_{x}u_{0}=0, \quad \nabla_{x}\mathcal{L}u_{0}=0, \quad \nabla_{x}\mathcal{L}^{2}w=0 \quad {\rm at } \quad O.
\end{equation*}

We observe that
\begin{align*}
&\partial_{\tau_{1}}^{3}w=\partial_{\tau_{1}}^{2}\partial_{\nu_{1}}w=0, \quad \partial_{\tau_{2}}^{3}w=\partial_{\tau_{2}}^{2}\partial_{\nu_{2}}w=0, \\
&\partial_{\tau_{1}}^{2}\partial_{\tau_{2}}w=-\partial_{\tau_{1}}^{2}\partial_{\nu_{1}}w=0, \quad \partial_{\tau_{1}}\partial_{\tau_{2}}^{2}w=\partial_{\tau_{1}}\partial_{\nu_{1}}^{2}w=0.
\end{align*}
Then the relation (\ref{eq5}) gives $\nabla_{x}^{3}w=0$ at $O$. Applying $\nabla_{x}$ to both sides of (\ref{eq}) gives $\nabla_{x}u_{0}(O)=0$ and $\nabla_{x}\mathcal{L}u_{0}(O)=0$. Going back to (\ref{eq}) and applying $\nabla_{x}\mathcal{L}$ to its both sides we have $\nabla_{x}\mathcal{L}^{2}w(O)=0$.

{\bf Step 4.} We show that
\begin{equation} \label{eq11}
\nabla_{x}^{4}w=0, \quad \nabla_{x}^{2}u_{0}=0, \quad \nabla_{x}^{2}\mathcal{L}u_{0}=0, \quad \nabla_{x}^{2}\mathcal{L}^{2}w=0 \quad {\rm at } \quad O.
\end{equation}

We only need to show $\nabla_{x}^{4}w=0$ at $O$ and the conclusion (\ref{eq11}) can be obtained as in the previous step. Obviously, we can get
\begin{equation*}
\partial_{\tau_{1}}^{4}w=\partial_{\tau_{1}}^{3}\partial_{\tau_{2}}w=\partial_{\tau_{1}} \partial_{\tau_{2}}^{3}w=\partial_{\tau_{2}}^{4}w=0 \quad {\rm at} \quad O.
\end{equation*}
Hence it suffices to prove $\partial_{\tau_{1}}^{2}\partial_{\tau_{2}}^{2}w=0$ at $O$.

Going back to the decomposition form of $\mathcal{L}^{2}=\Delta P_{1}(\partial) P_{2}(\partial)$, we have
\begin{equation*}
P_{1}(\partial)=\left(\begin{array}{cc}
-\partial_{\tau_{2}} & \partial_{\tau_{1}} \\
\partial_{\tau_{1}} & \partial_{\tau_{2}}
\end{array}\right)\left(\begin{array}{cc}
\mu & 0 \\  0 & \lambda+2\mu
\end{array}\right)
=\partial_{\tau_{1}}\left(\begin{array}{cc}
0 & \lambda+2\mu \\  \mu & 0
\end{array}\right)-\,\partial_{\tau_{2}}\left(\begin{array}{cc}
\mu & 0 \\  0 & -(\lambda+2\mu)
\end{array}\right),
\end{equation*}
\begin{equation*}
P_{2}(\partial)=\left(\begin{array}{cc}
\mu & 0 \\   0 & \lambda+2\mu
\end{array}\right)\left(\begin{array}{cc}
-\partial_{\tau_{2}} & \partial_{\tau_{1}} \\
\partial_{\tau_{1}} & \partial_{\tau_{2}}
\end{array}\right)
=\partial_{\tau_{1}}\left(\begin{array}{cc}
0 & \mu \\   \lambda+2\mu & 0
\end{array}\right)-\,\partial_{\tau_{2}}\left(\begin{array}{cc}
\mu & 0 \\   0 & -(\lambda+2\mu)
\end{array}\right).
\end{equation*}
Denote $\Lambda_1(\partial):=\partial_{\tau_1}-i\partial_{\tau_2}$, $\Lambda_2(\partial):= \partial_{\tau_1}+i\partial_{\tau_2}$. Then it is easy to check that
\begin{equation*}
\partial_{\tau_1}=\frac{1}{2}\big[\Lambda_1(\partial)+\Lambda_2(\partial)\big], \quad
\partial_{\tau_2}=\frac{i}{2}\big[\Lambda_1(\partial)-\Lambda_2(\partial)\big], \quad
\Delta=\Lambda_1(\partial)\Lambda_2(\partial),
\end{equation*}
and
\begin{align*}
P_{1}(\partial)=&\frac{\Lambda_1(\partial)}{2}\left(\begin{array}{cc}
-i\mu & \lambda+2\mu \\
\mu & i(\lambda+2\mu)
\end{array}\right)
+\frac{\Lambda_2(\partial)}{2}\left(\begin{array}{cc}
i\mu & \lambda+2\mu \\
\mu & -i(\lambda+2\mu)
\end{array}\right) \\
=&\frac{1}{2}\Big[\,\Lambda_{1}(\partial)B+\Lambda_{2}(\partial)\overline{B}\,\Big],
\end{align*}
\begin{align*}
P_{2}(\partial)=&\frac{\Lambda_1(\partial)}{2}\left(\begin{array}{cc}
-i\mu & \mu \\
\lambda+2\mu & i(\lambda+2\mu)
\end{array}\right)
+\frac{\Lambda_2(\partial)}{2}\left(\begin{array}{cc}
i\mu & \mu \\
\lambda+2\mu & -i(\lambda+2\mu)
\end{array}\right)\\
=&\frac{1}{2}\Big[\,\Lambda_{1}(\partial)B^{T}+\Lambda_{2}(\partial)\overline{B}^{T}\,\Big],
\end{align*}
where $\overline{B}$ is the conjugate matrix, $B^{T}$ is the transpose matrixes of $B$ defined by
\begin{equation*}
B:=\left(\begin{array}{cc}
-i\mu & \lambda+2\mu \\
\mu & i(\lambda+2\mu)
\end{array}\right).
\end{equation*}
Then we conclude that for all $m\in\mathbb{N}_{0}$,
\begin{equation} \label{c}
\nabla_{x}^{m}\subset{\rm span}\{\Lambda_{1}^{m}(\partial),\Lambda_{1}^{m-1}(\partial)\Lambda_{2}(\partial), \Lambda_{1}^{m-2}(\partial)\Lambda_{2}^{2}(\partial), \cdots, \Lambda_{1}(\partial)\Lambda_{2}^{m-1}(\partial),\Lambda_{2}^{m}(\partial)\},
\end{equation}
and
\begin{equation*}
\mathcal{L}^{2}=\frac{1}{4}\,\Lambda_1(\partial)\,\Lambda_2(\partial)
\big[\,\Lambda_{1}^{2}(\partial)BB^{T}+\Lambda_{1}(\partial)\Lambda_{2}(\partial) (B\overline{B}^{T}+\overline{B}B^{T})
+\Lambda_{2}^{2}(\partial)\overline{B}\,\overline{B}^{T}\,\big],
\end{equation*}
with $a:=(\lambda+2\mu)^{2}$, $b:=\mu^{2}$,
\begin{equation*}
BB^{T}=(a-b)\left(\begin{array}{cc}
1 & i \\ i & -1
\end{array}\right), \quad  \overline{B}\,\overline{B}^{T}=(a-b)\left(\begin{array}{cc}
1 & -i \\ -i & -1
\end{array}\right), \quad B\overline{B}^{T}+\overline{B}B^{T}=2(a+b)I.
\end{equation*}

In Step 2, $\mathcal{L}^{2}w(O)=0$ implies that
\begin{equation*}
\Lambda_{1}^{3}(\partial)\Lambda_2(\partial)BB^{T}w(O) +\Lambda_{1}^{2}(\partial)\Lambda_{2}^{2}(\partial)(B\overline{B}^{T}+\overline{B}B^{T})w(O)
+\Lambda_1(\partial)\Lambda_{2}^{3}(\partial)\overline{B}\,\overline{B}^{T}w(O)=0.
\end{equation*}
That is
\begin{equation*}
(a-b)\left[ \Lambda_{1}^{3}(\partial)\,\Lambda_2(\partial)(w_{1}+i w_{2}) +\Lambda_1(\partial)\,\Lambda_{2}^{3}(\partial)(w_{1}-iw_{2})\right]
+2(a+b) \Lambda_{1}^{2}(\partial)\,\Lambda_{2}^{2}(\partial)w_{1}=0  \,{\rm at}\, O,
\end{equation*}
\begin{equation*}
(a-b)\left[\Lambda_{1}^{3}(\partial)\,\Lambda_2(\partial)(iw_{1}-w_{2}) -\Lambda_1(\partial)\,\Lambda_{2}^{3}(\partial)(iw_{1}+w_{2})\right]
+2(a+b) \Lambda_{1}^{2}(\partial)\,\Lambda_{2}^{2}(\partial)w_{2}=0  \,{\rm at}\, O,
\end{equation*}
where $w:=(w_{1},w_{2})^{\top}$. Simple calculation yields
\begin{equation*}
(a-b)\left[\Lambda_1(\partial)\,\Lambda_{2}^{3}(\partial)(w_{1}-iw_{2})\right]
+(a+b) \Lambda_{1}^{2}(\partial)\,\Lambda_{2}^{2}(\partial)(w_{1}+iw_{2})=0\quad\,{\rm at} \quad O,
\end{equation*}
\begin{equation*}
(a-b)\left[\Lambda_{1}^{3}(\partial)\,\Lambda_2(\partial)(w_{1}+iw_{2})\right]
+(a+b)\Lambda_{1}^{2}(\partial)\,\Lambda_{2}^{2}(\partial)(w_{1}-iw_{2})=0\quad \,{\rm at} \quad O.
\end{equation*}
Set $W_{1}:=w_{1}-iw_{2}$, $W_{2}:=w_{1}+iw_{2}$, that is
\begin{equation} \label{req1}
(a-b)\Lambda_1(\partial)\,\Lambda_{2}^{3}(\partial)W_{1}
+(a+b) \Lambda_{1}^{2}(\partial)\,\Lambda_{2}^{2}(\partial)W_{2}=0 \quad \,{\rm at} \quad O,
\end{equation}
\begin{equation} \label{req2}
(a+b)\Lambda_{1}^{2}(\partial)\,\Lambda_{2}^{2}(\partial)W_{1} +(a-b)\Lambda_{1}^{3}(\partial)\,\Lambda_2(\partial)W_{2}=0 \quad \,{\rm at} \quad O.
\end{equation}

In Step 4, the conclusion $\partial_{\tau_{1}}^{4}w=\partial_{\tau_{1}}^{3}\partial_{\tau_{2}}w =\partial_{\tau_{1}}\partial_{\tau_{2}}^{3}w=\partial_{\tau_{2}}^{4}w=0$ at $O$ implies that
\begin{equation*}
\big(\Lambda_1(\partial)+\Lambda_2(\partial)\big)^{4}w(O)= \big(\Lambda_1(\partial)+\Lambda_2(\partial)\big)^{3} \big(\Lambda_1(\partial)-\Lambda_2(\partial)\big)w(O)=0,
\end{equation*}
\begin{equation*}
\big(\Lambda_1(\partial)+\Lambda_2(\partial)\big) \big(\Lambda_1(\partial)-\Lambda_2(\partial)\big)^{3}w(O) =\big(\Lambda_1(\partial)-\Lambda_2(\partial)\big)^{4}w(O)=0.
\end{equation*}
That is,
\begin{equation*}
\big(\Lambda_1(\partial)+\Lambda_2(\partial)\big)^{3}\Lambda_1(\partial)W_{t}(O)= \big(\Lambda_1(\partial)+\Lambda_2(\partial)\big)^{3}\Lambda_2(\partial)W_{t}(O)=0,
\end{equation*}
\begin{equation*}
\big(\Lambda_1(\partial)-\Lambda_2(\partial)\big)^{3}\Lambda_1(\partial)W_{t}(O) =\big(\Lambda_1(\partial)-\Lambda_2(\partial)\big)^{3}\Lambda_2(\partial)W_{t}(O)=0.
\end{equation*}
Transform the above equations into a $4\times4$ system
\begin{equation*}
\left(\begin{array}{c} \Lambda^{4}_1(\partial) \\ \Lambda^{3}_1(\partial)\Lambda_2(\partial) \\ \Lambda_1(\partial)\Lambda^{3}_2(\partial) \\  \Lambda^{4}_2(\partial) \end{array}\right) W_{t}(O)= \left(\begin{array}{cccc}
8 & 0 & 8 & 0  \\    -3 & -1 & 3 & -1  \\    1 & 3 & -1 & 3  \\  0 & -8 & 0 & 8
\end{array}\right) \left(\begin{array}{c} 3 \\ -3 \\ 3 \\ 3  \end{array}\right) \frac{\Lambda^{2}_1(\partial)\Lambda^{2}_2(\partial)}{-16}W_{t}(O), \quad t=1,2,
\end{equation*}
we obtain that
\begin{equation*}
\Lambda^{3}_1(\partial)\Lambda_2(\partial)W_{t}(O)=0, \quad \quad
\Lambda_1(\partial) \Lambda^{3}_2(\partial)W_{t}(O)=0.
\end{equation*}
Combining with the relations in \eqref{req1} and \eqref{req2}, we achieve that $\Lambda^{2}_{1} (\partial)\Lambda_{2}^{2}(\partial)W_{t}(O)=0$ ($t=1,2$). Consequently,
\begin{equation*}
\Lambda^{4}_1(\partial)w=\Lambda^{3}_1(\partial)\Lambda_2(\partial)w=\Lambda^{2}_{1} (\partial)\Lambda_{2}^{2}(\partial)=\Lambda_1(\partial) \Lambda^{3}_2(\partial)w= \Lambda^{4}_2(\partial)w=0 \quad \,{\rm at} \quad O.
\end{equation*}
Then the inclusion relation \eqref{c} gives that $\nabla_{x}^{4}w=0$ at $O$.

{\bf Step 5.} We show that
\begin{equation} \label{eq12}
\nabla_{x}^{5}w=0, \quad \nabla_{x}^{3}u_{0}=0, \quad \nabla_{x}^{3}\mathcal{L}u_{0}=0, \quad \nabla_{x}^{3}\mathcal{L}^{2}w=0 \quad {\rm at } \quad O.
\end{equation}

We only need to show $\nabla_{x}^{5}w=0$ at $O$ and the conclusion (\ref{eq12}) can be obtained as in the previous step. By the relation \eqref{c} and the definitions of $W_{t}$, we need to show that
\begin{equation*}
\Lambda_{1}^{m}(\partial)\Lambda_{2}^{5-m}(\partial)W_{t}= 0\quad (m=0,1,2,3,4,5; ~~ t=1,2) \quad {\rm at} \quad O.
\end{equation*}
Obviously, we can get from Step 4 that
\begin{equation*}
\partial_{\tau_{1}}^{5}w=\partial_{\tau_{1}}^{4}\partial_{\tau_{2}}w=\partial_{\tau_{1}} \partial_{\tau_{2}}^{4}w=\partial_{\tau_{2}}^{5}w=0 \quad {\rm at} \quad O,
\end{equation*}
which implies that
\begin{equation} \label{d}
\left(\begin{array}{c} \Lambda^{5}_1(\partial) \\ \Lambda^{4}_1(\partial)\Lambda_2(\partial) \\ \Lambda_1(\partial)\Lambda^{4}_2(\partial) \\  \Lambda^{5}_2(\partial) \end{array}\right) W_{t}= \sum_{j=2}^{3}\left(\begin{array}{cccc}
4 & 1 & 4 & -1  \\  -1 & 0 & 1 & 0  \\  0 & -1 & 0 & 1  \\  1 & 4 & -1 & 4
\end{array}\right) \left(\begin{array}{c}
C_{4}^{5-j}(-1)^{5-j} \\ C_{4}^{4-j}(-1)^{4-j} \\ C_{4}^{5-j} \\ C_{4}^{4-j} \end{array}\right) \frac{\Lambda^{j}_1(\partial)\Lambda^{5-j}_2(\partial)}{-8}W_{t}.
\end{equation}
Consequently, it is sufficient to prove that $\Lambda_{1}^{2}(\partial)\Lambda_{2}^{3}(\partial)W_{t} =\Lambda_{1}^{3}(\partial)\Lambda_{2}^{2}(\partial)W_{t}=0$ at $O$.

In step 3, $\nabla_{x}\mathcal{L}^{2}w(O)=0$ implies that $\partial_{\tau_{1}} \mathcal{L}^{2}w(O)=0$, $\partial_{\tau_{2}}\mathcal{L}^{2}w(O)=0$ and then
\begin{equation*}
\Lambda_{1}(\partial)\mathcal{L}^{2}w(O)=0,\quad\Lambda_{2}(\partial)\mathcal{L}^{2}w(O)=0.
\end{equation*}
Recalling (\ref{req1}) and (\ref{req2}), the above relations tell us that
\begin{equation} \label{req1-4}
(a-b)\Lambda_{1}^{2}(\partial)\,\Lambda_{2}^{3}(\partial)W_{1}(O)
+(a+b) \Lambda_{1}^{3}(\partial)\,\Lambda_{2}^{2}(\partial)W_{2}(O)=0,
\end{equation}
\begin{equation} \label{req2-4}
(a-b)\Lambda_{1}(\partial)\,\Lambda_{2}^{4}(\partial)W_{1}(O)
+(a+b) \Lambda_{1}^{2}(\partial)\,\Lambda_{2}^{3}(\partial)W_{2}(O)=0,
\end{equation}
\begin{equation} \label{req3-4}
(a+b)\Lambda_{1}^{3}(\partial)\,\Lambda_{2}^{2}(\partial)W_{1}(O) +(a-b)\Lambda_{1}^{4}(\partial)\,\Lambda_{2}(\partial)W_{2}(O)=0,
\end{equation}
\begin{equation} \label{req4-4}
(a+b)\Lambda_{1}^{2}(\partial)\,\Lambda_{2}^{3}(\partial)W_{1}(O) +(a-b)\Lambda_{1}^{3}(\partial)\,\Lambda_{2}^{2}(\partial)W_{2}(O)=0.
\end{equation}
Transform \eqref{req1-4} with \eqref{req4-4} into a 2-by-2 system yields,
\begin{equation*}
\left|\begin{array}{cc}
a-b & a+b  \\    a+b & a-b \end{array}\right|=-4ab\neq0, \quad {\rm then}  \quad
\left\{\begin{array}{l}
\Lambda_{1}^{2}(\partial)\,\Lambda_{2}^{3}(\partial)W_{1}(O)=0 , \\
\Lambda_{1}^{3}(\partial)\,\Lambda_{2}^{2}(\partial)W_{2}(O)=0.
\end{array}\right.
\end{equation*}
By the representation system \eqref{d}, we have
\begin{equation*}
\Lambda^{4}_1(\partial)\Lambda_2(\partial)W_{t}(O) =-\Lambda^{2}_1(\partial)\Lambda_{2}^{3}(\partial)W_{t}(O), \quad \quad
\Lambda_{1}(\partial)\Lambda^{4}_2(\partial)W_{t}(O) =-\Lambda^{3}_1(\partial)\Lambda_{2}^{2}(\partial)W_{t}(O).
\end{equation*}
Insert the relations into \eqref{req2-4} and \eqref{req3-4} yields
\begin{equation*}
\left\{\begin{array}{c}
-(a-b)\Lambda^{3}_1(\partial)\Lambda_{2}^{2}(\partial)W_{1}(O)
+(a+b) \Lambda_{1}^{2}(\partial)\,\Lambda_{2}^{3}(\partial)W_{2}(O)=0, \\
(a+b)\Lambda_{1}^{3}(\partial)\,\Lambda_{2}^{2}(\partial)W_{1}(O) -(a-b)\Lambda^{2}_1(\partial)\Lambda_{2}^{3}(\partial)W_{2}(O)=0.
\end{array}\right.
\end{equation*}
That is,
\begin{equation*}
\Lambda^{3}_1(\partial)\Lambda_{2}^{2}(\partial)W_{1}(O) =\Lambda^{2}_1(\partial)\Lambda_{2}^{3}(\partial)W_{2}(O)=0.
\end{equation*}

{\bf Step 6.} We show that
\begin{equation*}
\nabla_{x}^{6}w=0, \quad \nabla_{x}^{4}u_{2}=0, \quad \nabla_{x}^{4}\mathcal{L}u_{2}=0, \quad \nabla_{x}^{4}\mathcal{L}^{2}w=0 \quad {\rm at } \quad O.
\end{equation*}
Similarly, we only need to show $\nabla_{x}^{6}w=0$ at $O$, which is equivalent to
\begin{equation*}
\Lambda_{1}^{m}(\partial)\Lambda_{2}^{6-m}(\partial)W_{t}= 0\quad (m=0,1,2,3,4,5,6; ~~ t=1,2) \quad {\rm at} \quad O.
\end{equation*}
Since the relations $\partial_{\tau_{1}}^{6}w=\partial_{\tau_{1}}^{5}\partial_{\tau_{2}}w =\partial_{\tau_{1}}\partial_{\tau_{2}}^{5}w=\partial_{\tau_{2}}^{6}w=0$ at $O$ obtained from Step 5 implies that
\begin{equation} \label{e}
\left(\begin{array}{c} \Lambda^{6}_1(\partial) \\ \Lambda^{5}_1(\partial)\Lambda_2(\partial) \\ \Lambda_1(\partial)\Lambda^{5}_2(\partial) \\  \Lambda^{6}_2(\partial)\end{array}\right)W_{t}= \sum_{j=2}^{4}\left(\begin{array}{cccc}
24 & 0 & 24 & 0  \\  -5 & -1 & 5 & -1  \\ 1 & 5 & -1 & 5  \\ 0 & -24 & 0 & 24
\end{array}\right) \left(\begin{array}{c}
C_{5}^{6-j}(-1)^{6-j} \\ C_{5}^{5-j}(-1)^{5-j} \\ C_{5}^{6-j} \\ C_{5}^{5-j} \end{array}\right) \frac{\Lambda^{j}_1(\partial)\Lambda^{6-j}_2(\partial)}{-48}W_{t},
\end{equation}
it suffices to prove that $\Lambda^{j}_1(\partial)\Lambda^{6-j}_2(\partial)W_{t}=0$ ($j=2,3,4; \, t=1,2$) at $O$.

In step 4, $\nabla^{2}_{x}\mathcal{L}^{2}w(O)=0$ implies that $\nabla^{2}_{\tau} \mathcal{L}^{2}w(O)=0$ and then
\begin{equation*}
\Lambda^{2}_{1}(\partial)\mathcal{L}^{2}w(O)=0,\quad \Lambda^{2}_{2}(\partial)\mathcal{L}^{2}w(O)=0, \quad \Lambda_{1}(\partial)\Lambda_{2}(\partial)\mathcal{L}^{2}w(O)=0.
\end{equation*}
From (\ref{req1-4}) to (\ref{req4-4}), we have
\begin{equation} \label{req1-5}
(a-b)\Lambda_{1}^{3}(\partial)\,\Lambda_{2}^{3}(\partial)W_{1}
+(a+b) \Lambda_{1}^{4}(\partial)\,\Lambda_{2}^{2}(\partial)W_{2}=0,
\end{equation}
\begin{equation} \label{req2-5}
(a-b)\Lambda_{1}^{2}(\partial)\,\Lambda_{2}^{4}(\partial)W_{1}
+(a+b) \Lambda_{1}^{3}(\partial)\,\Lambda_{2}^{3}(\partial)W_{2}=0,
\end{equation}
\begin{equation} \label{req4-5}
(a-b)\Lambda_{1}(\partial)\,\Lambda_{2}^{5}(\partial)W_{1}
+(a+b) \Lambda_{1}^{2}(\partial)\,\Lambda_{2}^{4}(\partial)W_{2}=0,
\end{equation}
\begin{equation} \label{req5-5}
(a+b)\Lambda_{1}^{4}(\partial)\,\Lambda_{2}^{2}(\partial)W_{1} +(a-b)\Lambda_{1}^{5}(\partial)\,\Lambda_{2}(\partial)W_{2}=0,
\end{equation}
\begin{equation} \label{req6-5}
(a+b)\Lambda_{1}^{3}(\partial)\,\Lambda_{2}^{3}(\partial)W_{1} +(a-b)\Lambda_{1}^{4}(\partial)\,\Lambda_{2}^{2}(\partial)W_{2}=0,
\end{equation}
\begin{equation} \label{req8-5}
(a+b)\Lambda_{1}^{2}(\partial)\,\Lambda_{2}^{4}(\partial)W_{1} +(a-b)\Lambda_{1}^{3}(\partial)\,\Lambda_{2}^{3}(\partial)W_{2}=0.
\end{equation}
Combining \eqref{req1-5} with \eqref{req6-5}, \eqref{req2-5} with \eqref{req8-5} yields
\begin{equation*}
\Lambda_{1}^{3}(\partial)\,\Lambda_{2}^{3}(\partial)W_{1} =\Lambda_{1}^{4}(\partial)\,\Lambda_{2}^{2}(\partial)W_{2}=0 \quad {\rm and} \quad
\Lambda_{1}^{2}(\partial)\,\Lambda_{2}^{4}(\partial)W_{1}= \Lambda_{1}^{3}(\partial)\,\Lambda_{2}^{3}(\partial)W_{2}=0.
\end{equation*}
Recall the representation system \eqref{e}, we have
\begin{equation*}
\Lambda^{5}_1(\partial)\Lambda_2(\partial)W_{t} =-\frac{5}{3}\Lambda^{3}_1(\partial)\Lambda_{2}^{3}(\partial)W_{t}=0, \quad \quad
\Lambda_{1}(\partial)\Lambda^{5}_2(\partial)W_{t} =-\frac{5}{3}\Lambda^{3}_1(\partial)\Lambda_{2}^{3}(\partial)W_{t}=0.
\end{equation*}
Insert the above results into \eqref{req4-5} and \eqref{req5-5}, we conclude that
\begin{equation*}
\Lambda_{1}^{4}(\partial)\,\Lambda_{2}^{2}(\partial)W_{1} =\Lambda_{1}^{2}(\partial)\,\Lambda_{2}^{4}(\partial)W_{2}=0.
\end{equation*}

{\bf Step 7.} For $m\geq 6$, we make the induction hypothesis that for $0\leq k\leq m$,
\begin{equation*}
\nabla_x^k w=0, \quad \nabla_x^{k-2} u_0=0, \quad \nabla_x^{k-2} \mathcal{L}u_0=0, \quad
\nabla_x^{k-2} \mathcal{L}^2 w=0 \quad {\rm at} \quad O.
\end{equation*}
We show that
\begin{equation*}
\nabla_x^{m+1} w=0, \quad \nabla_x^{m-1} u_0=0, \quad \nabla_x^{m-1} \mathcal{L}u_0=0, \quad
\nabla_x^{m-1} \mathcal{L}^2 w=0 \quad {\rm at} \quad O.
\end{equation*}

From $\nabla_x^{m+1} w(O)=0$ and (\ref{eq}) we can get $\nabla_x^{m-1} u_0(O)=0$. Then using (\ref{eq}) yields $\nabla_x^{m-1} \mathcal{L}u_0(O)=0$, $\nabla_x^{m-1}\mathcal{L}^2 w(O)=0$. Thus, we only need to verify that $\nabla_x^{m+1}w(O)=0$, i.e.,
\begin{equation*}
\partial_{\tau_1}^{m+1}w=\partial_{\tau_1}^{m} \partial_{\tau_2}w= \cdots =\partial_{\tau_1} \partial_{\tau_2}^{m}w=\partial_{\tau_2}^{m+1}w=0 \quad {\rm at} \quad O,
\end{equation*}
which is equivalent to
\begin{equation*}
\Lambda_{1}^{m+1}(\partial)W_{t}=\Lambda_{1}^{m}(\partial)\Lambda_{2}(\partial)W_{t}= \cdots=\Lambda_{1}(\partial) \Lambda_{2}^{m}(\partial)W_{t}=\Lambda_{2}^{m+1}(\partial)W_{t}=0 \,~(t=1,2) \quad {\rm at} \quad O.
\end{equation*}
Obviously, we can get
\begin{equation*}
\partial_{\tau_{1}}^{m+1}w=\partial_{\tau_{1}}^{m}\partial_{\tau_{2}}w=\partial_{\tau_{1}} \partial_{\tau_{2}}^{m}w=\partial_{\tau_{2}}^{m+1}w=0 \quad {\rm at} \quad O,
\end{equation*}
which implies that
\begin{equation} \label{W1}
\big[\Lambda_1(\partial)+\Lambda_2(\partial)\big]^{m}\Lambda_1(\partial)W_{t}  =\big[\Lambda_1(\partial)+\Lambda_2(\partial)\big]^{m}\Lambda_2(\partial)W_{t} =0 \quad {\rm at} \quad O,
\end{equation}
\begin{equation} \label{W2}
\big[\Lambda_1(\partial)-\Lambda_2(\partial)\big]^{m}\Lambda_1(\partial)W_{t}  =\big[\Lambda_1(\partial)-\Lambda_2(\partial)\big]^{m}\Lambda_2(\partial)W_{t} =0 \quad {\rm at} \quad O.
\end{equation}
For the case that $m$ is an even number, relations (\ref{W1}) and (\ref{W2}) give that
\begin{equation*}
\left(\begin{array}{c} \Lambda^{m+1}_1(\partial) \\ \Lambda^{m}_1(\partial)\Lambda_2(\partial) \\ \Lambda_1(\partial)\Lambda^{m}_2(\partial) \\  \Lambda^{m+1}_2(\partial) \end{array}\right)W_{t}(O)= \sum_{j=2}^{m-1}\left(\begin{array}{cccc}
m & 1 & m & -1  \\  -1 & 0 & 1 & 0  \\ 0 & -1 & 0 & 1  \\ 1 & m & -1 & m
\end{array}\right) \alpha(m,j) \,
\frac{\Lambda^{j}_{1}(\partial)\Lambda^{m+1-j}_2(\partial)}{-2m}W_{t}(O),
\end{equation*}
with
\begin{equation*}
\alpha(m,j):=\left(\begin{array}{llll}
C_{m}^{m+1-j}(-1)^{m+1-j}, & C_{m}^{m-j}(-1)^{m-j}, & C_{m}^{m+1-j}, & C_{m}^{m-j} \end{array}\right)^{\top},
\end{equation*}
and especially
\begin{equation*}
\Lambda^{m}_1(\partial)\Lambda_2(\partial)W_{t}(O)=-\frac{1}{2m}\sum_{j=2}^{m-1} \Big[C_{m}^{m+1-j}+(-1)^{m-j}C_{m}^{m+1-j} \Big]\Lambda^{j}_{1}(\partial)\Lambda^{m+1-j}_2(\partial)W_{t}(O),
\end{equation*}
\begin{equation*}
\Lambda_1(\partial)\Lambda_2^{m}(\partial)W_{t}(O)=-\frac{1}{2m}\sum_{j=2}^{m-1} \Big[C_{m}^{m-j}-(-1)^{m-j}C_{m}^{m-j} \Big]\Lambda^{j}_{1}(\partial)\Lambda^{m+1-j}_2(\partial)W_{t}(O).
\end{equation*}
For the case that $m$ is an odd number, relations (\ref{W1}) and (\ref{W2}) give that
\begin{equation*}
\left(\begin{array}{c} \Lambda^{m+1}_1(\partial) \\ \Lambda^{m}_1(\partial)\Lambda_2(\partial) \\ \Lambda_1(\partial)\Lambda^{m}_2(\partial) \\  \Lambda^{m+1}_2(\partial) \end{array}\right)W_{t}(O)= \sum_{j=2}^{m-1} Q(m)\, \alpha(m,j) \, \frac{\Lambda^{j}_{1}(\partial)\Lambda^{m+1-j}_2(\partial)}{-2(m^{2}-1)}W_{t}(O),
\end{equation*}
with
\begin{equation*}
Q(m):=\left(\begin{array}{cccc}
m^{2}-1 & 0 & m^{2}-1 & 0 \\  -m & -1 & m & -1  \\ 1 & m & -1 & m \\ 0 & 1-m^{2} & 0 & m^{2}-1
\end{array}\right),
\end{equation*}
and especially
\begin{align*}
\Lambda^{m}_1(\partial)\Lambda_2(\partial)W_{t}(O)=-\frac{1}{2(m^{2}-1)}\sum_{j=2}^{m-1} &\Big[mC_{m}^{m+1-j}+(-1)^{m-j}mC_{m}^{m+1-j}  \\
&-C_{m}^{m-j}-(-1)^{m-j}C_{m}^{m-j} \Big] \Lambda^{j}_{1}(\partial)\Lambda^{m+1-j}_2(\partial)W_{t}(O),
\end{align*}
\begin{align*}
\Lambda_1(\partial)\Lambda_2^{m}(\partial)W_{t}(O)=-\frac{1}{2(m^{2}-1)}\sum_{j=2}^{m-1} & \Big[mC_{m}^{m-j}+(-1)^{m-j}mC_{m}^{m-j} \\
-&C_{m}^{m+1-j}+(-1)^{m+1-j}C_{m}^{m+1-j} \Big] \Lambda^{j}_{1}(\partial)\Lambda^{m+1-j}_2(\partial)W_{t}(O).
\end{align*}
Consequently, whatever $m$ is even or odd, it is sufficient to show that
\begin{equation*}
\Lambda^{j}_{1}(\partial)\Lambda^{m+1-j}_2(\partial)W_{t}(O)=0 \quad {\rm for}\quad  j=2,3,\cdots,m-1.
\end{equation*}

In the induction hypothesis, $\nabla^{m-3}_{x}\mathcal{L}w(O)=0$ (i.e. $\nabla^{m-3}_{\tau}\mathcal{L}^{2}w(O)=0$) implies that
\begin{equation*}
\Lambda_{1}^{m-3}(\partial)\mathcal{L}^{2}w =\Lambda_{1}^{m-4}(\partial)\Lambda_{2}(\partial)\mathcal{L}^{2}w =\cdots =\Lambda_{1}(\partial)\Lambda^{m-4}_{2}(\partial)\mathcal{L}^{2}w =\Lambda_{2}^{m-3}(\partial)\mathcal{L}^{2}w=0 \quad {\rm at} \quad O.
\end{equation*}
Going back to the decomposition of $\mathcal{L}^{2}$ in Step 4 and repeating the process in gaining equations \eqref{req1}, \eqref{req2}, the above relations tell us that
\begin{equation*}
\left\{\begin{array}{l}
(a-b)\Lambda_{1}^{j}(\partial)\Lambda_{2}^{m+1-j}(\partial)W_{1}(O) +(a+b)\Lambda_{1}^{j+1}(\partial)\Lambda_{2}^{m-j}(\partial)W_{2}(O)=0,  \\
(a+b)\Lambda_{1}^{j}(\partial)\Lambda_2^{m+1-j}(\partial)W_{1}(O) +(a-b)\Lambda_{1}^{j+1}(\partial)\Lambda_{2}^{m-j}(\partial)W_{2}(O)=0,
\end{array}\right.
\end{equation*}
for $j=2,3,\cdots,m-2$, and
\begin{equation} \label{lasteq}
\left\{\begin{array}{l}
(a-b)\Lambda_1(\partial)\Lambda_2^{m}(\partial)W_{1}(O) +(a+b)\Lambda_{1}^{2}(\partial)\Lambda_{2}^{m-1}(\partial)W_{2}(O)=0,  \\
(a+b)\Lambda_{1}^{m-1}(\partial)\Lambda_2^{2}(\partial)W_{1}(O) +(a-b)\Lambda_{1}^{m}(\partial)\Lambda_{2}(\partial)W_{2}(O)=0.
\end{array}\right.
\end{equation}
Note that the determinant of coefficient matrix is not equal to zero (i.e. $\left|\begin{array}{cc} a-b & a+b  \\    a+b & a-b \end{array}\right|=-4ab\neq0$), we obatin that
\begin{equation*}
\Lambda_{1}^{2}(\partial)\Lambda_2^{m-1}(\partial)W_{1} =\cdots=\Lambda_{1}^{m-2}(\partial)\Lambda_2^{3}(\partial)W_{1}=0 \quad {\rm at}\quad O,
\end{equation*}
\begin{equation*}
\Lambda_{1}^{m-1}(\partial)\Lambda_2^{2}(\partial)W_{2} =\cdots=\Lambda_{1}^{3}(\partial)\Lambda_2^{m-2}(\partial)W_{2}=0 \quad {\rm at}\quad O.
\end{equation*}
For the remainding terms, we bring the above results into the expressions of $\Lambda^{m}_1(\partial) \Lambda_2(\partial)W_{2}(O)$ and $\Lambda_1(\partial)\Lambda_2^{m}(\partial)W_{1}(O)$, that is: when $m$ is even, we have
\begin{equation*}
\Lambda^{m}_1(\partial)\Lambda_2(\partial)W_{2}=-\frac{1}{2m} \Big(C_{m}^{m-1}+C_{m}^{m-1} \Big) \Lambda^{2}_{1}(\partial)\Lambda^{m-1}_2(\partial)W_{2} =-\Lambda^{2}_{1}(\partial)\Lambda^{m-1}_2(\partial)W_{2} \quad {\rm at}\,~ O,
\end{equation*}
\begin{equation*}
\Lambda_1(\partial)\Lambda_2^{m}(\partial)W_{1}=-\frac{1}{2m} \Big(C_{m}^{1}+C_{m}^{1} \Big)\Lambda^{m-1}_{1}(\partial)\Lambda^{2}_2(\partial) W_{1} =-\Lambda^{m-1}_{1}(\partial)\Lambda^{2}_2(\partial)W_{1} \quad {\rm at}\,~ O;
\end{equation*}
when $m$ is odd, we have
\begin{align*}
\Lambda^{m}_1(\partial)\Lambda_2(\partial)W_{2}=-\frac{1}{2(m^{2}-1)} \Big(mC_{m}^{m-1}-mC_{m}^{m-1} -C_{m}^{m-2}+C_{m}^{m-2} \Big) \Lambda^{2}_{1}(\partial)\Lambda^{m-1}_2(\partial)W_{2}=0,
\end{align*}
\begin{equation*}
\Lambda_1(\partial)\Lambda_2^{m}(\partial)W_{1}=-\frac{1}{2(m^{2}-1)} \Big(mC_{m}^{1}-mC_{m}^{1}-C_{m}^{2}+C_{m}^{2} \Big)\Lambda^{m-1}_{1}(\partial)\Lambda^{2}_2(\partial)W_{1}=0 \quad {\rm at}\,~ O.
\end{equation*}
Combining these expresentations with \eqref{lasteq} yields
\begin{equation*}
\Lambda_1(\partial)\Lambda_2^{m}(\partial)W_{1} =\Lambda^{m-1}_{1}(\partial)\Lambda^{2}_2(\partial)W_{1} =\Lambda^{m}_1(\partial)\Lambda_2(\partial)W_{2} =\Lambda^{2}_{1}(\partial)\Lambda^{m-1}_2(\partial)W_{2}=0 \quad {\rm at}\,~O.
\end{equation*}

By induction, we have proved that $\nabla_x^m w(O)=\nabla_x^m u_0(O)=0$ for all $m\in \mathbb{N}_0$. Since $w$ and $u_0$ are analytic in $B_{R}$, we have $w=u_0\equiv0$ in $B_{R}$. The proof is complete.
\end{proof}
\begin{remark}
The proof of Lemma \ref{mainlem} relies heavily on the singularity analysis around a rectangular corner, which together with the polar coordinates has simplified our arguments. We believe that the results of this work extend to general corners with an arbitrary opening angle $\varphi\in (0, 2\pi)\backslash \{\pi\}$.
\end{remark}

\section{Proof of Main Theorem and Corollaries} \label{sec4}
{\bf Proof of Theorem \ref{Main}}. Suppose on the contrary that $u^{{\rm in}}$ can be extended as an incident wave (i.e a solution to \eqref{b}) to an open domain around the corner.
Without loss of generality, we assume that the corner is located at origin. Since $u^{{\rm sc}}\in[H_{0}^{2}(D)]^{2}$, we extend it to $B_{R}$ as a function which is zero outside $D$, i.e.
\begin{equation*}
\left\{\begin{array}{ll}
\mathcal{L}u^{{\rm in}}+\omega^{2}u^{{\rm in}}=0 & \quad {\rm in} \quad B_{R},  \\
\mathcal{L}u^{{\rm sc}}+\omega^{2}\rho_{0}u^{{\rm sc}}=\omega^{2}(1-\rho_{0})u^{{\rm in}}  &\quad {\rm in} \quad B_{R},  \\
u^{{\rm sc}}=Tu^{{\rm sc}}=0 & \quad {\rm on} \quad B_{R}\cap\partial D.
\end{array}\right.
\end{equation*}
By Lemma \ref{mainlem}, we obtain that $u^{{\rm in}}=u^{{\rm sc}}=0$ in $B_{R}$. And then $u^{{\rm in}}=u^{{\rm sc}}=0$ in $D$, which contradicts that the pair $(u^{{\rm in}},u^{{\rm sc}})$ are interior transmission eigenfunctions. The proof is complete.

{\bf Proof of Corollary \ref{cor1}}. If the kernel of the far-field operator $F$ is nontrivial, then there is a Herglotz wavefunction $u_{g}$ satisfying \eqref{b} in $\mathbb{R}^{2}$, and an outgoing $u^{{\rm sc}}$ with vanishing far-field. Rellich's lemma and unique continuation guarantee that $u^{{\rm sc}}$ vanishes outside the support of $1-\rho(x)$, that is, $u^{{\rm sc}}\equiv0$ in $\mathbb{R}^{2}\backslash \overline{D}$. It follows from the fact $\rho(x)\in L^{\infty}(\mathbb{R}^{2})$ and $u^{{\rm in}}\in[L^{2}(D)]^{2}$ that $u^{{\rm sc}}\in[H^{2}_{{\rm loc}}(\mathbb{R}^{2})]^{2}$, and therefore the restriction of $u^{{\rm sc}}$ and its first derivative to $\partial D$ must vanish. Hence the pair $(u^{{\rm sc}},u_{g})$ are interior transmission eigenfunctions in $D$, but $u_{g}$ extends past the corner, contradicting Theorem \ref{Main}.

{\bf Proof of Corollary \ref{cor2}}. Let $u_{1}$ and $u_{2}$ be solutions to the direct scattering problem \eqref{a}-\eqref{rad} corresponding to $(D_{1},\rho_{1})$ and $(D_{2},\rho_{2})$, respectively. Assume that $u_{1}^{\infty}(\hat{x}) =u_{2}^{\infty}(\hat{x})$ for all $\hat{x}\in\mathbb{S}$, then Rellich's lemma and unique continuation guarantee that $u_{1}^{{\rm sc}}(x)=u_{2}^{{\rm sc}}(x)$ for all $x\in\mathbb{R}^{2}\backslash (\overline{D_{1}\cup D_{2}})$.

If $D_{1}\neq D_{2}$, without loss of generality, we may assume that there exists a corner $O$ of $\partial D_{1}$ such that $O\not\in\partial D_{2}$. Since this corner stays away from $D_{2}$, the function $u_{2}$ satisfies the Navier equation with the frequency $\omega^{2}$ around $O$, while $u_{1}$ fulfills the Navier equation with the variable potential $\rho_{1}\omega^{2}$, that is,
\begin{equation*}
\left\{\begin{array}{ll}
\mathcal{L} u_{1}+\rho_{1}\omega^{2} u_{1}=0 &\quad {\rm in}~~B_{R}, \\
\mathcal{L} u_{2}+\omega^{2} u_{2}=0 &\quad {\rm in}~~B_{R}, \\
u_{1}=u_{2},\quad Tu_{1}=Tu_{2}  &\quad {\rm on}~~B_{R}\cap\partial D_{1}.
\end{array}\right.
\end{equation*}
Here, $R$ is sufficient small such $B_{R}\cap D_{2}=\emptyset$ and $\rho_{1}$ is a piecewise constant in $B_{R}$, i.e. $\rho_{1}=1$ in $B_{R}\backslash\overline{D_{1}}$, $\rho_{1}= c_{0}\neq 1$ in $B_{R}\cap D_{1}$. Since $u_{2}$ is analytic in $B_{R}$, the data of $u_{1}$, $Tu_{1}$ on $B_{R} \cap\partial D_{1}$ are analytic by the transmission boundary conditions. By the Cauchy-Kowalewski theorem and Holmgren's theorem, we can find a solution $\widetilde{u}_{1}$ to the following problem in a piecewise analytic domain:
\begin{equation*}
\left\{\begin{array}{ll}
\mathcal{L} \widetilde{u}_{1}+c_{0}\omega^{2} \widetilde{u}_{1}=0 &\quad {\rm in}~~B_{\widetilde{R}}\backslash\overline{D_{1}}, \\
\widetilde{u}_{1}=u_{1},\quad T\widetilde{u}_{1}=Tu_{1}  &\quad {\rm on}~~B_{\widetilde{R}}\cap\partial D_{1},
\end{array}\right.
\end{equation*}
for some $0<\widetilde{R}<R$. Setting $w=u_{1}-u_{2}$ in $B_{\widetilde{R}}\cap D_{1}$ and $w= \widetilde{u}_{1}-u_{2}$ in $B_{\widetilde{R}}\backslash\overline{D_{1}}$ yields
\begin{equation*}
\left\{\begin{array}{ll}
\mathcal{L}w+c_{0}\omega^{2}w=\omega^{2}(1-c_{0})u_{2} & \quad {\rm in} \quad B_{\widetilde{R}},  \\
\mathcal{L}u_{2}+\omega^{2}u_{2}=0  &\quad {\rm in} \quad B_{\widetilde{R}},  \\
w=Tw=0 & \quad {\rm on} \quad B_{\widetilde{R}}\cap\partial D_{1}.
\end{array}\right.
\end{equation*}
Now, applying lemma \ref{mainlem}, we obtain the vanishing of $u_{2}$ near $O$ and thus $u_{2} \equiv0$ (or equivalently, $u_{2}^{{\rm sc}}=-u^{{\rm in}}$) in $\mathbb{R}^{2}\backslash \overline{D_{2}}$ by the unique continuation. This implies that the radiating solution $u_{2}^{{\rm sc}}$ is also an entire function in $\mathbb{R}^{2}$. Hence, $u_{2}^{{\rm sc}} =u^{{\rm in}}\equiv0$ in $\mathbb{R}^{2}$, contradicting our assumption on the incident wave. Thus, $D_{1}=D_{2}$.

\vspace{0.3cm}
{\bf Acknowledgments.}

The work of G.H. Hu is partially supported by the National Natural Science Foundation of China (No. NSFC 12425112) and the Fundamental Research Funds for Central Universities in China (No. 050-63213025). The work of J.L. Xiang is supported by the Natural Science Foundation of China (No. 12301542), the Open Research Fund of Hubei Key Laboratory of Mathematical Sciences (Central China Normal University, MPL2025ORG017) and the China Scholarship Council. Part of this work was based on discussions with Dr. T. Yin from AMSS, CAS, China, whom are greatly acknowledged.

\end{document}